\documentclass[11pt, reqno]{amsart}
\usepackage{amsmath, amssymb, latexsym, enumerate,  graphicx, MnSymbol}
\usepackage{amscd,mathrsfs,epic,empheq,float}
\usepackage{bbm}
\usepackage[all]{xy}
 \usepackage{pdfsync}
 \usepackage{todonotes}

\usepackage{tikz}
\usetikzlibrary{arrows}
\usepackage{graphicx}
\usepackage[vcentermath]{youngtab}
\usetikzlibrary{decorations.markings}
\usetikzlibrary{fadings}
\usepackage{array}
\usepackage{multirow}
\usepackage{stmaryrd}
\usepackage[latin1]{inputenc}

\theoremstyle{definition}
\newtheorem{thm}{Theorem}[section]
\newtheorem{prop}[thm]{Proposition}

\newtheorem{lem}{Lemma}

\newtheorem{rem}{Remark}
\newtheorem{ex}[thm]{Example}

\newtheorem*{first-step}{First Step}
\newtheorem*{third-step}{Third Step}
\newtheorem*{second-step}{Second Step}

\newtheorem*{main idea}{Main idea}

\newtheorem{cl}{Claim}

\newcommand{\op}{\operatorname}

\newcommand{\tightoverset}[2]{%
  \mathop{#2}\limits^{\vbox to -.5ex{\kern-0.75ex\hbox{$#1$}\vss}}}
  
\newcommand{\tightunderset}[2]{%
  \mathop{#2}\limits_{\vbox to -.5ex{\kern-0.75ex\hbox{$#1$}\vss}}}

\newcommand\smvee{\raise0.9ex\hbox{$\scriptscriptstyle\vee$}}

\newcommand{\C}{{\rm C}}
\newcommand{\T}{{\rm T}}
\newcommand{\E}{{\rm E}}

\newcommand{\U}{{\rm U}}
\newcommand{\mm}{{\rm H}}
\newcommand{\B}{{\rm B}}

\newcommand{\M}{{\rm M}}

\newcommand{\e}{\varepsilon}

\newcommand{\A}{{\rm A}}
\newcommand{\F}{{\rm F}}

\newcommand{\li}{{\rm L}}
\newcommand{\X}{{\rm X}}
\newcommand{\Q}{{\rm Q}}

\newcommand{\Z}{{\rm Z}}

\newcommand{\G}{{\rm G}}
\newcommand{\<}{\langle}
\newcommand{\rr}{\rangle}

\date{}                                           
\newcount\cols
{\catcode`,=\active\catcode`|=\active
 \gdef\Young#1{\hbox{$\vcenter
 {\mathcode`,="8000\mathcode`|="8000
  \def,{\global\advance\cols by 1 &}%
  \def|{\cr
        \multispan{\the\cols}\hrulefill\cr
        &\global\cols=2 }%
  \offinterlineskip\everycr{}\tabskip=0pt
  \dimen0=\ht\strutbox \advance\dimen0 by \dp\strutbox
  \halign
   {\vrule height \ht\strutbox depth \dp\strutbox##
    &&\hbox to \dimen0{\hss$##$\hss}\vrule\cr
    \noalign{\hrule}&\global\cols=2 #1\crcr
    \multispan{\the\cols}\hrulefill\cr%
   }
 }$}}
\gdef\Skew(#1:#2){\hbox{$\vcenter
{\mathcode`,="8000\mathcode`|="8000
  \dimen0=\ht\strutbox \advance\dimen0 by \dp\strutbox
  \def\boxbeg{\vbox
    \bgroup\hrule\kern-0.4pt\hbox to\dimen0\bgroup\strut\vrule\hss$}%
  \def\boxend{$\hss\egroup\hrule\egroup}%
  \def,{\boxend\boxbeg}%
  \def|##1:{\boxend\vrule\egroup\nointerlineskip\kern-0.4pt
    \moveright##1\dimen0\hbox\bgroup\boxbeg}%
  \def\\##1\\##2:{\boxend\vrule\egroup\nointerlineskip\kern-0.4pt
    \kern ##1\dimen0\moveright##2\dimen0\hbox\bgroup\boxbeg}%
  \moveright#1\dimen0\hbox\bgroup\boxbeg#2\boxend\vrule\egroup
 }$}}
}

\title{Word reading is a crystal morphism}
\author{Jacinta Torres}

\dedicatory{To my friend Bea}

\begin{document}
\maketitle
\begin{abstract}
We observe that word reading is a crystal morphism. This leads us to prove that for $\op{SL}_{n}(\mathbb{C})$ the map from \textit{all} galleries to Mikovi\'c Vilonen cycles  is a surjective morphism of crystals. We also compute the fibers of this map in terms of the Littelmann path model.  
\end{abstract}

\section{Introduction}

Both the Littelmann path model \cite{pathmodel} and the set of Mikovi\'c Vilonen (MV) cycles \cite{mirkovicvilonen} \cite{bravermangaitsgory} give constructions of the crystal associated to a simple module of a connected reductive group $\G$ over the field of complex numbers $\mathbb{C}$. The path model consists of paths in the real vector space spanned by the weight lattice, and the elements of the set of MV cycles are certain closed subsets of the affine Grassmannian $\mathcal{G}$ of the group $\G^{\vee}$ that is Langlands dual to $\G$.\\

By consdering piecewise linear paths contained in the one-skeleton of the standard apartment in the affine building \cite{ronan} of $\G^{\vee}$ and interpreting them as \textsl{one-skeleton galleries}, Gaussent and Littelmann assigned a closed subset of the affine Grassmannian $\mathcal{G}$ \cite{ls} \cite{onesk} to each of these piecewise linear paths. They showed that if the path is LS then the closed subset associated to it is an MV cycle. (LS paths were introduced by Lakshmibai and Seshadri \cite{lsoriginal} and were the first paths to be studied \cite{lspaths}.) This association defines a bijection which was shown to be an isomorphism of crystals by Baumann and Gaussent \cite{baumanngaussent}. \\

We work with the special linear group $\G = \op{SL}_{n}(\mathbb{C})$. In this case piecewise linear paths are parametrised by combinatorial arrangements which we call \textsl{galleries} - with respect to this identification the set of LS paths corresponds to the set of semistandard Young tableaux with columns of maximal length $n-1$. In this setting, Gaussent, Littelmann and Nguyen showed that the closed subset associated to any gallery is an MV cycle \cite{knuth}. To do this they considered the monoid $\mathcal{W}_{n}$ of words in the alphabet $\mathcal{A}_{n}:= \{1, \cdots, n\}$ and the associated plactic monoid $\mathcal{P}_{n} = \mathcal{W}_{n}/ \sim$ which is defined as the quotient of $\mathcal{W}_{n}$ by the ideal $\sim$ generated by the following relations.

\begin{itemize}
\label{a}
\item[a.] For $x\leq y < z$, $y\hbox{ }x\hbox{ }z = y\hbox{ }z\hbox{ }x$.
\label{b}
\item[b.] For $x< y \leq z$, $x\hbox{ }z\hbox{ }y = x\hbox{ }y\hbox{ }z$.
\label{c}
\item[c.] The relation $1\hbox{ }\cdots \hbox{ } n = \o$, where $\o$ is the trivial word. 
\end{itemize}

Relations a. and b. are the well-known \textsl{Knuth relations} \cite{knuthrelations}. A gallery $\gamma$, and in particular a semi-standard Young tableau, defines a word $w(\gamma)$ in $\mathcal{W}_{n}$. The classes in the plactic monoid are in bijection with the set of words of semi-standard Young tableaux. To associate an MV cycle to any gallery the authors of \cite{knuth} show that the closed subset associated to a gallery $\gamma$ depends only on the class $[w(\gamma)] \in \mathcal{P}_{n}$ of its word in the plactic monoid. (Actually relation c. was overlooked in \cite{knuth}. In the appendix (Appendix \ref{appendix}) we show that the closed subset associated to two words related by it stays the same.)\\

Crystals, however, are not mentioned in \cite{knuth}. In this paper we show that their map is a surjective morphism of crystals and determine its fibers (Theorem \ref{main}). To do so we observe that, considering words as galleries, the map that assigns the word $w(\gamma)$ to the gallery $\gamma$ is a morphism of crystals (Proposition \ref{wordreadingisacrystalmorphism}). As a direct consequence we obtain that it is an isomorphism onto its image when restricted to each connected component.

\subsection*{Acknowledgements}
The author would like to thank Peter Littelmann for introducing her to the topic, St\'ephane Gaussent for useful discussions, Michael Ehrig and Bea Schumann for their comments, and Daniel Juteau for his helpful suggestions. The author would also like to thank both referees for their time and their skilful comments - and for an observation that led to writing the appendix. The author has been supported by the Graduate School 1269: Global structures in geometry and analysis - financed by the Deutsche Forschungsgemeinschaft; she has also been partially supported by the SPP1388. 

\section{Galleries, words and crystals}
\subsection{Galleries and their words}
The combinatorics developed in this section is related to the representation theory of the group $\op{SL}_{n}(\mathbb{C})$, where $n \in \mathbb{Z}_{>0}$ is a fixed positive integer. Throughout this paper, all representations will be representations of $\op{SL}_{n}( \mathbb{C})$. A \textit{shape} is a finite sequence of positive integers $\underline{d} = (d_{1}, \cdots, d_{r})$, each $d_{s}$ less than or equal to $n-1$. An arrangement of boxes of shape $\underline{d}$ is  an arrangement of $r$ columns of boxes such that column $s$ (read from right to left) has $d_{s}$ boxes.

\begin{ex}
An arrangement of boxes of shape (1,1,2,1).
\begin{align*}
\Skew(0:,,,|1:) 
\end{align*}
\end{ex}

\noindent
A \textit{gallery} of shape $\underline{d}$ is a filling of an arrangement of boxes of the given shape with letters from the ordered alphabet $\mathcal{A}_{n} := \{1, \cdots, n: 1< \cdots< n\}$ such that entries are strictly increasing along each column of boxes. We will denote the set of galleries of shape $\underline{d}$ by $\Gamma(\underline{d})$, the set of all galleries by $\Gamma$, and, given a gallery $\gamma$, we will denote its shape by $\underline{d}(\gamma)$.

\begin{ex}
A gallery of shape (1,1,2,1).
\begin{align*}
\Skew(0:\mbox{\tiny{3}},\mbox{\tiny{1}},\mbox{\tiny{5}},\mbox{\tiny{2}}|1:\mbox{\tiny{2}}) 
\end{align*}
\end{ex}

Let $\mathcal{W}_{n}$ denote the word monoid on $\mathcal{A}_{n}$. To a word $ w = a_{1}\hbox{ } \cdots \hbox{ } a_{k} \in \mathcal{W}_{n}$ is associated the gallery $\gamma_{a_{1}\hbox{ } \cdots \hbox{ } a_{k}}=\gamma_{w} := \Skew(0:a_{k}, \cdots, a_{1})$. The \textit{word} of a gallery of shape $(m)$ (this means it has a single column of length $m$) is the word in $\mathcal{W}_{n}$ that corresponds to reading its entries from top to bottom and writing them down from left to right. We will sometimes call galleries of shape $(m)$ \textsl{column} galleries. The word of an arbitrary gallery is the concatenation of the words of each of its columns read from right to left - concatenation of two galleries $\gamma_{2}* \gamma_{1}$ is done starting with $\gamma_{1}$ from the right. We denote the word of a gallery $\delta$ by $w(\delta)$. Note that if $w' \in \mathcal{W}_{n}$ is a word, then $w(\gamma_{w'}) = w'$.

\begin{ex}
The galleries $\beta = \Skew(0:\mbox{\tiny{3}},\mbox{\tiny{2}},\mbox{\tiny{1}},\mbox{\tiny{5}},\mbox{\tiny{2}})$ and $\gamma = \Skew(0:\mbox{\tiny{3}},\mbox{\tiny{1}},\mbox{\tiny{5}},\mbox{\tiny{2}}|1:\mbox{\tiny{2}})$ both have word 25123 $= w(\gamma) = w(\beta)$. 

\end{ex}

\subsection{Characters, cocharacters, weights, and coweights}
\label{charactersandweightssection}

In this section we recall some basic facts and establish some notation. First consider the group $\op{GL}_{n}(\mathbb{C})$ of invertible $n\times n$ matrices, and in it the maximal torus $\T_{\op{GL}_{n}}(\mathbb{C})$ of diagonal matrices. Then maximal tori for $\op{SL}_{n}(\mathbb{C}) = [\op{GL}_{n}(\mathbb{C}),\op{GL}_{n}(\mathbb{C})]$ and $\op{PSL}_{n}(\mathbb{C}) = \op{GL}_{n}(\mathbb{C})/ \mathbb{C}^{\times}\op{Id}$ are given by $\T_{\op{SL}_{n}}(\mathbb{C}):= \T_{\op{GL}_{n}}(\mathbb{C})\cap \op{SL}_{n}(\mathbb{C})$ and $\T_{\op{PSL}_{n}}(\mathbb{C}):= \op{can}(\T_{\op{GL}_{n}}(\mathbb{C}))$ respectively, where $\op{can}: \op{GL}_{n}(\mathbb{C}) \rightarrow \op{PSL}_{n}(\mathbb{C})$ is the canonical map.\\

 We want to look at paths in $\mathbb{V}:= \X\otimes_{\mathbb{Z}}\mathbb{R}$, where $\X = \X(\T_{\op{SL}_{n}}(\mathbb{C}))= \op{Hom}(\T_{\op{SL}_{n}}(\mathbb{C}),\mathbb{C}^{\times})$ is the set of characters of $\T_{\op{SL}_{n}}(\mathbb{C})$ and the corresponding full weight lattice. For this consider $\mathbb{R}^{n}$ with inner product $(-,-)$ and orthonormal basis $\{\varepsilon_{1}, \cdots, \varepsilon_{n}\}$. Then $ \mathbb{V} \cong \{w \in \mathbb{R}^{n}: (w,e_{1}+\cdots+e_{n}) = 0\} \cong \mathbb{R}^{n}/\mathbb{R}(e_{1}+\cdots + e_{n})$ and we make the following identifications: 
 
\begin{align*}
\X^{\vee} = \X^{\vee}(\T_{\op{SL}_{n}}(\mathbb{C}))) &= \op{Hom}(\mathbb{C}^{\times}, \T_{\op{SL}_{n}}(\mathbb{C})) \\
&= \{a_{1}\e_{1}+ \cdots + a_{n}\e_{n}: a_{i} \in \mathbb{Z}  ; \sum_{i = 1}^{n} a_{i} = 0\}= \mathbb{Z}\Phi^{\vee}, \\
\X = \X(\T_{\op{SL}_{n}}(\mathbb{C}))&= \op{Hom}(\T_{\op{SL}_{n}}(\mathbb{C}),\mathbb{C}^{\times}) = \bigoplus_{i=1}^{n}\mathbb{Z}\e_{i}/ \<\sum_{i=1}^{n} \e_{i}\rr \cong  \\
 \op{Hom}(\mathbb{C}^{\times}, \T_{\op{PSL}_{n}}(\mathbb{C})) &= \X^{\vee}(\T_{\op{PSL}_{n}}(\mathbb{C}))), \hbox{ and }\\
\X(\T_{\op{PSL}_{n}}(\mathbb{C})))& =  \op{Hom}(\T_{\op{PSL}_{n}}(\mathbb{C}),\mathbb{C}^{\times}) \\
&= \{a_{1}\e_{1}+ \cdots + a_{n}\e_{n}: a_{i} \in \mathbb{Z}; \sum_{i = 1}^{n} a_{i} = 0\}= \mathbb{Z}\Phi,
\end{align*}

where $\Phi$ and $\Phi^{\vee}$ are the sets of roots and coroots, respectively. The inner product $(-,-)$ restricts to the pairing between $\X$ and $\X^{\vee}$. In particular the root data $(\X(\T_{\op{SL}_{n}}(\mathbb{C}))), \X^{\vee}(\T_{\op{SL}_{n}}(\mathbb{C}))), \Phi, \Phi^{\vee})$ associated to\\ $(\op{SL}_{n}(\mathbb{C}), \T_{\op{SL}_{n}}(\mathbb{C})))$ is dual to that $(\X(\T_{\op{PSL}_{n}}(\mathbb{C}))), \X^{\vee}(\T_{\op{PSL}_{n}}(\mathbb{C}))), \Phi^{\vee}, \Phi)$ of $(\op{PSL}_{n}(\mathbb{C}), \T_{\op{PSL}_{n}}(\mathbb{C})))$. We choose the set of simple roots $\Delta = \{\alpha_{i} = \e_{i}-\e_{i+1}: 1 \leq i < n\}$, which in this case coincides with the corresponding set of simple coroots $\alpha_{i}^{\vee} = \alpha_{i} \in \Delta^{\vee}$. We write $\Phi^{+}$ and $\Phi^{\vee, +}$ for the corresponding sets of positive roots and coroots, respectively. The corresponding $i$-th fundamental weight is $\omega_{i} = \e_{1}+\cdots+ \e_{i}$, for $i \in \{1, \cdots, n-1\}$.  We will also consider the following hyperplane and half-spaces associated to a pair $(\alpha, n) \in \Phi \times \mathbb{Z}$:

\begin{align*}
\mm_{\alpha,n} = \{x \in \mathbb{V}: (\alpha,x) = n\}\\
\mm_{\alpha,n}^{+} = \{x \in \mathbb{V}: (\alpha,x) \geq n\}\\
\mm_{\alpha,n}^{+} = \{x \in \mathbb{V}: (\alpha,x) \leq n\}.
\end{align*}
\noindent
The dominant Weyl chamber is identified with the intersection $\bigcap_{\alpha_{i}\in \Delta}\mm_{\alpha_{i}, 0}$.

\subsection{Littelmann paths}
\label{litpathsection}
Each gallery defines a piecewise linear path in $\mathbb{V} \cong \X \otimes_{\mathbb{Z}}\mathbb{R}$ as follows. To a column with entries the integers $0 \leq l_{1}<\cdots <l_{k} \leq n$ we associate the path $\pi:[0,1]\rightarrow \mathbb{V}, t \mapsto t(\varepsilon_{l_{1}}+\cdots +\varepsilon_{l_{k}})$. The path associated to a gallery $\delta$ is defined to be the concatenation of the paths of its columns, beginning with the right-most one, just as when reading the word. We will denote it by $\pi_{\delta}$. A gallery $\delta$ is \textsl{dominant} if the image $\pi_{\delta}([0,1])$ of its corresponding path is contained in the dominant Weyl chamber. See Theorems \ref{thepathmodel} and \ref{pathmodel} below for the representation-theoretic meaning of paths. 

\begin{ex}
\label{example2}
Let $n = 3$. In the picture below (the shaded region is the dominant Weyl chamber), we see that the gallery $\nu = \Skew(0:\mbox{\tiny{1}},\mbox{\tiny{1}}|1:\mbox{\tiny{2}})$ is dominant while $\delta = \Skew(0:\mbox{\tiny{2}},\mbox{\tiny{3}},\mbox{\tiny{1}})$ is not. Note that $\pi_{\nu} = \pi_{\tiny{\hbox{$\Skew(0:\mbox{\tiny{1}}|0:\mbox{\tiny{2}})$}}}* \pi_{\hbox{\tiny{$\Skew(0:\mbox{\tiny{1}})$}}}$ and 
$\pi_{\hbox{\tiny{$\Skew(0:1)$}}}* \pi_{\hbox{\tiny{$\Skew(0:3)$}}}* \pi_{\hbox{\tiny{$\Skew(0:2)$}}}$.

\begin{center}
\begin{tikzpicture}
\draw[-, dotted] (-2,0) --(2,0);
\draw[-, dotted] (-2,0.86) --(2,0.86);
\draw[-, dotted] (-2,-0.86) --(2,-0.86);
\node[right] at (0.5, 0.86602540378){$\varepsilon_{1}$};
\draw[->, dotted] (0,0) --(0.5, 0.86602540378);
\fill[gray!20!, path fading=north] (0,0) --  (1,1.73205080756) -- (-1,1.73205080756) -- cycle;
\draw[->,dotted] (0,0) --(-1, 0);
\draw[-,dotted] (1,-1.73205080756) --(-1, 1.73205080756);
\draw[-,dotted] (2,-1.73205080756) --(0, 1.73205080756);
\draw[-,dotted] (0,-1.73205080756) --(-2, 1.73205080756);
\node[left] at (-1, 0){$\varepsilon_{2}$};
\draw[->, dotted] (0,0) --(0.5, -0.86602540378);
\draw[-, dotted] (1,1.73205080756) --(-1, -1.73205080756);
\draw[-, dotted] (2,1.73205080756) --(0, -1.73205080756);
\draw[-, dotted] (0,1.73205080756) --(-2, -1.73205080756);
\node[below] at (0.5, -0.86602540378) {$\varepsilon_{3}$};
\draw[-, blue] (0,0) --(0.5, 0.86602540378);
\draw[-, blue] (0.5, 0.86602540378) --(1,0);
\draw[-, blue] (1,0) -- (0,0);
\node[right] at (1, 0) {$\pi_{\delta}$\small{([0,1])} };
\draw[-, red] (0,0) --(-0.5, 0.86602540378);
\draw[-, red] (-0.5, 0.86602540378) -- (0,1.73205080756);
\node[above] at (0, 1.73205080756) {$\pi_{\nu}$\small{([0,1])} };
\end{tikzpicture}
\end{center}
\end{ex}

\begin{rem}
 The paths associated to galleries are examples of Littelmann paths, see \cite{pathmodel} and Theorem \ref{pathmodel} below. The images $\pi([0,1])$ of these paths are \textsl{one-skeleton galleries} in the standard apartment of the affine building of type A. This is explained in \cite{onesk}.
\end{rem}

\subsection{Crystals and representation theory}
We recall the crystal structure on the set of all galleries. We refer to \cite{bravermangaitsgory} and \cite{kashiwaraoncrystalbases}. For this section only, let $(\X, \X^{\vee}, \Phi, \Phi^{\vee})$ be a root datum,  $\G$ the corresponding complex reductive group, and $\Delta = \{\alpha_{i}: i \in \{1, \cdots, n-1\}\}$ a choice of simple roots.\\

A \textsl{crystal} is a set $\B$ of \textsl{vertices} together with maps
\begin{align*}
 e_{\alpha_i}, f_{\alpha_i}:& \B \rightarrow \B \cup \{0\} \mbox{  (the \textsl{root operators})},\\ \hbox{ and }\op{wt}: &\B \rightarrow \X 
\end{align*} 
 for each $i \in \{1, \cdots, n-1\}$ such that for every $b, b' \in \B$ and $i\in \{1, \cdots , n-1\}, b' = e_{\alpha_i}(b)$ if and only if $b = f_{\alpha_i}(b')$, and, in this case, setting $\epsilon_{\alpha_i}(b''):= \op{max}\{n: e_{\alpha_i}^{n}(b)\neq 0 \}$ and $\phi_{\alpha_i}(b''):=\op{max}\{n: f_{\alpha_i}^{n}(b'')\neq 0\}$ for any $b'' \in \B$, the following properties are satisfied.

\begin{enumerate}
\item $\op{wt}(b') = \op{wt}(b)+\alpha_{i}$
\item $\phi(b) = \epsilon_{\alpha_i}(b) +(\op{wt}(b), \alpha_{i}^{\vee})$
\end{enumerate}

\noindent
A crystal is in particular a graph, and is hence a disjoint union of its connected components. If $\B$ is a crystal and $b \in \B$ is a vertex we will denote the connected component of $\B$ in which it lies by $\op{Conn}(b)$. A \textsl{crystal morphism} is a map $\F: \B \rightarrow \B'$ between the underlying sets of two crystals $\B$ and $\B'$ such that $\op{wt}(\F(b)) = \op{wt}(b)$ and such that it commutes with the action of the root operators. A crystal morphism is an isomorphism if it is bijective. Given an integrable module $M$ of the quantum group $U_{q}(\mathfrak{g})$ of the Lie algebra $\mathfrak{g}$ of $\G$, Kashiwara constructed a crystal $\B_{M}$ that is the ``combinatorial skeleton'' of $M$ \cite{kashiwaraoncrystalbasesoriginal}. If $M= \li(\lambda)$ is a simple module of highest weight $\lambda \in \X^{+}$ then $\B_{M}$ is a connected crystal denoted by $\B(\lambda)$, which has the property that there exists a unique element $b_{\lambda} \in \B(\lambda)$ such that $e_{i}b_{\lambda} = 0$ for all $i \in \{1, \cdots , n-1\}$.  Such an element is called a \textsl{highest weight vertex}. The crystal $\B(\lambda)$ also has the characterising property that $\op{dim}(\li(\lambda)_{\mu}) = \#\{b \in \B(\lambda): \op{wt}(b) = \mu\}$. If $M = \bigoplus_{i = 1}^{m}\li(\lambda_{i})$ is semisimple, then the connected components of $\B_{M}$ contain exactly one highest weight vertex of $\B_{M}$ each. They are in one-to-one correspondence with the crystals $\B(\lambda_{i})$ that correspond to the simple summands $\li(\lambda_{i})$ of $\M$.

\subsection{Crystal structure on the set of galleries}
Let $\gamma$ be a gallery of shape $\underline{d} = (d_{1}, \cdots, d_{r} )$. Define $\op{wt}(\gamma):= \pi_{\gamma}(1) \in \X$. Note that this is well defined. In Example \ref{example2}, $\op{wt}(\nu) = 2\epsilon_{1}+\epsilon_{2}$ and $\op{wt}(\delta) = \epsilon_{1}+\epsilon_{2} +\varepsilon_{3} = 0$. In general $\op{wt}(\gamma) = \sum_{i=1}^{n}\lambda_{i}\epsilon_{i}$, where $\lambda_{i} := \#\{i's \hbox{ in } \gamma \}$. We recall the action of the root operator $f_{\alpha_i}$ (respectively $e_{\alpha_i}$) on $\gamma$. The definition we provide here is a straightforward generalization of the crystal operators on Young tableaux given in \cite{kashiwaraoncrystalbases} and a translation of the crystal operators on paths  \cite{pathmodel} or galleries \cite{ls}. See also Section 7.4 of \cite{hongandkang}.

\subsubsection{The action of the root operators $e_{\alpha_{i}}, f_{\alpha_{i}}$}
\begin{itemize}
\item[a.] Tag the columns of $\gamma$ with a sign $\sigma \in \{+,-, \emptyset\}$ in the following way. If both $i$ and $i+1$ appear in a column or if they do not appear, the column is tagged with a $(\emptyset)$. If only $i$ appears, it is tagged with a $(+)$, and if only $i+1$ appears, with a $(-)$. The resulting sequence of tags is sometimes called the \textit{i-signature} of $\gamma$.

\item[b.] Ignore the $(\emptyset)$-tagged columns to produce a sub-gallery, and then ignore all pairs of consecutive columns tagged $(-\hbox{ } +)$, and get another sub-gallery. Continue this process, recursively obtaining sub-galleries, until a final sub-gallery is produced with tags of the form
$$(+)^{s}(-)^{r}.$$
To apply the operator $f_{\alpha_i}$ (resp. $e_{\alpha_i}$), modify the column corresponding to the right most $(+)$ (resp. left most $(-)$) in the final sub-gallery tags, and replace the entry $i$ with $i+1$ (resp. $i+1$ with $i$). If $s = 0$ (resp. $r = 0$), then $f_{\alpha_i}(\gamma) = 0$ (resp. $e_{\alpha_i}(\gamma) = 0$). 
\end{itemize}

It is easy to check that the above operations define a crystal structure on the set of galleries $\Gamma$.

\begin{ex}
To apply the crystal operator $f_{\alpha_2}$ to
\begin{align*}
\gamma = \Skew(0:\mbox{\tiny{3}},\mbox{\tiny{1}},\mbox{\tiny{5}},\mbox{\tiny{2}}|1:\mbox{\tiny{2}}),
\end{align*}
one obtains that the corresponding taggings of the columns read from left to right are $-+\emptyset+$. The first sub-gallery obtained is 

\begin{align*}
\Skew(0:\mbox{\tiny{3}},\mbox{\tiny{1}}, \mbox{\tiny{2}}|1:\mbox{\tiny{2}}),
\end{align*}
which is tagged by $-++$. The next sub-gallery is then $\Skew(0:2)$, hence 
\begin{align*}
f_{\alpha_2}(\gamma) = \Skew(0:\mbox{\tiny{3}},\mbox{\tiny{1}},\mbox{\tiny{5}},\mbox{\tiny{3}}|1:\mbox{\tiny{2}}) .
\end{align*}
We also obtain that $f_{\alpha_1}(\gamma) = 0$.
\end{ex}

\subsection{Word reading and paths}
\subsubsection{The path model}
We begin this section with what is known as the Littelmann path model. Theorem \ref{thepathmodel} below is proven (in a more general context) as Theorem 7.1 in \cite{pathmodel}.

\begin{thm}
\label{thepathmodel}
If $\gamma \in \Gamma$ is a dominant gallery, then $\op{Conn}(\gamma) \cong \B(\op{wt}(\gamma))$. 
\end{thm}

\subsubsection{Word reading}
The following proposition is very important for our purposes. It is well known for semistandard Young tableaux (see for example \cite{kashiwaraoncrystalbases}, Section 5.3). Let $\underline{d} = (d_{1}, \cdots, d_{r})$ be a shape, $l_{d} = \sum_{j=1}^{r}d_{j}$ the number of boxes in the arrangement of boxes of shape $\underline{d}$ and $\underline{l_{d}} = \underbrace{(1,\cdots, 1)}_{l_{d}-times}$.

\begin{prop}
\label{wordreadingisacrystalmorphism}
The map
\begin{align*}
\Gamma(\underline{d}) &\longrightarrow \Gamma(\underline{l_{d}})\\
\gamma &\longmapsto \gamma_{w(\gamma)}
\end{align*}
 is a crystal morphism.
\end{prop}

\begin{proof}
First note that since the weight of a gallery only depends on the entries of its boxes, $\op{wt}(\gamma) = \op{wt}(\gamma_{w(\gamma)})$. If two single column galleries $\gamma_{1}, \gamma_{2}$ are labelled by $(+)$ and $(-)$ respectively, then the word associated to their concatenation $\gamma_{2}*\gamma_{1}$ is in turn labelled by $(-\hbox{ }+)$. If the gallery $\gamma$ is not labelled, then $\gamma_{w(\gamma)}$  is labelled either by $(-\hbox{ }+)$ or by $\emptyset$. It is therefore enough to show that for any $i\in \{1, \cdots, n-1\}$ and any gallery $\gamma$ of shape $(m), f_{\alpha_i}(w(\gamma)) = w(f_{\alpha_i}(\gamma))$. This is shown in \cite{kashiwaraoncrystalbases}, Section 5.3, Proposition 5.1. We give a proof nevertheless, for the comfort of the reader.\\
 
 Let $\gamma$ be a column gallery of shape $(m)$ with entries $1 \leq a_{1}< \cdots < a_{m} \leq n$ and $i \in \{1, \cdots, n-1\}$. If $\gamma$ is labelled by $(\emptyset)$ or by $(-)$ then $f_{\alpha_i}(w(\gamma)) = w(f_{\alpha_i}(\gamma)) = 0$. If $\gamma$ is labelled by $(+)$, then, for some $k \in \{1, \cdots, r\}$, $a_{k} = i$ and since the column is labelled by only a $(+), a_{k+1}> a_{k}+1$. Hence, $f_{\alpha_i}(\gamma)$ is obtained from $\gamma$ by replacing $i = a_{k}$ by $i+1$, with no need of reordering the entries, and therefore $f_{\alpha_i}(w(\gamma)) = w(f_{\alpha_i}(\gamma))$.
\end{proof}

Proposition \ref{wordreadingisacrystalmorphism} allows an enhanced version of Theorem \ref{thepathmodel} which we state in Theorem \ref{pathmodel} (it is well-known but the author has not found an explicit reference). To prove it we need the following lemma which characterizes dominant galleries as highest weight vertices. 

\begin{lem}
\label{highestweightvertex}
A gallery $\nu \in \Gamma$ is dominant if and only if $e_{\alpha_{i}}(\nu) = 0$ for all $i \in \{1, \cdots, n-1\}$. 
\end{lem}

\begin{proof}
Let $\nu \in \Gamma$ be a gallery. First notice the following two things.
\begin{itemize}
\item[1.] Since entries are strictly increasing in columns, the gallery $\nu$ is dominant if and only if $\gamma_{w(\nu)}$ is dominant.
\item[2.] For a word $w \in \mathcal{W}_{n}$, the condition $e_{i}(\gamma_{w}) = 0$ for all $i \in \{1, \cdots, n-1\}$ means that to the right of each $i+1$ in $\gamma_{w}$ is at least one $i$. This is equivalent to $\gamma_{w}$ being dominant. 
\end{itemize}
Now assume that $e_{i}(\nu) = 0$ for all $i \in \{1, \cdots, n\}$. By Proposition \ref{wordreadingisacrystalmorphism} this is equivalent to $e_{i}(\gamma_{w(\nu)}) = 0$ for all $i \in \{1, \cdots, n\}$, which by 2. above is equivalent to $\gamma_{w(\nu)}$ being dominant, which is in turn equivalent to $\nu$ being dominant by 1. above. 

\end{proof}

\begin{thm}[The type A path model]
\label{pathmodel}
The connected components of $\Gamma$ are all of the form $\op{Conn}(\delta) \cong \B(\op{wt}(\delta))$ for a dominant gallery $\delta$. 
\end{thm}

\begin{proof}
By Theorem \ref{thepathmodel} it is enough to show that for every gallery $\nu$ there is a dominant gallery $\delta \in \op{Conn}(\nu)$ that belongs to the same connected component as $\nu$. To see this consider a gallery $\nu \in \Gamma(\underline{d})$  of shape $\underline{d}$. Its word, seen as the gallery $\gamma_{w(\nu)}$, lies in the crystal $\Gamma(\underline{l_{d}})$. As is explained in Section 13 of \cite{placticalgebra}, this is the crystal $\B_{M}$ associated to the representation $\M :=\li(\omega_{1})^{\otimes l(w(\nu))}$, where $l(w(\nu))$ is the length of the word $w(\nu)$. The representation $\M$ is semisimple, hence $\gamma_{w(\nu)}$ lies in a connected component $ \op{Conn}(\gamma_{w(\nu)}) \cong \B(\lambda)$ isomorphic to the crystal associated to a simple module $\li(\lambda)$ of highest weight $\lambda \in \X^{+}$, with highest vertex $b_{\lambda} \in \Gamma(\underline{l_{d}})$. Proposition \ref{wordreadingisacrystalmorphism} implies that $\op{Conn}(\nu) \cong \op{Conn}(\gamma_{w(\nu)})$ - hence there exists a gallery $\delta \in \Gamma(\underline{d})$ such that $\gamma_{w(\delta)} = b_{\lambda}$. In particular, since $\gamma_{w(\delta)}$ is a highest weight vertex, by Lemma \ref{highestweightvertex} it is dominant, hence by 1. in the proof of Lemma \ref{highestweightvertex}, so is $\delta$.
\end{proof}

\begin{ex}
\label{wordreadingexample}
 A connected crystal of galleries of shape $(2,1)$ and the crystal formed by its word-readings, regarded as galleries, in the case $n=3$. Both crystals are isomorphic to  the crystal $\B(\omega_{1}+\omega_{2})$ associated to the simple module $\li(\omega_{1}+\omega_{2})$ for $\op{SL}_{3}(\mathbb{C})$. 
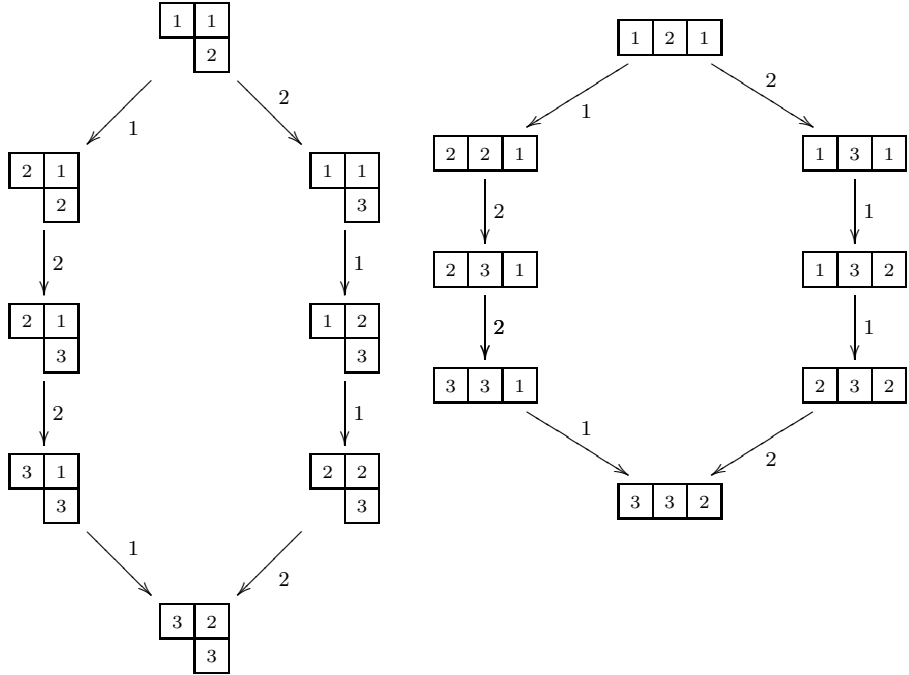
\begin{align*}
\xymatrix{
&\mbox{\Skew(0:\mbox{\tiny{1}},\mbox{\tiny{1}}|1:\mbox{\tiny{2}})}\ar^{1}[dl]\ar^{2}[dr]&\\
\mbox{\Skew(0:\mbox{\tiny{2}},\mbox{\tiny{1}}|1:\mbox{\tiny{2}})}\ar^{2}[d]&&\mbox{\Skew(0:\mbox{\tiny{1}},\mbox{\tiny{1}}|1:\mbox{\tiny{3}})}\ar^{1}[d]\\
\mbox{\Skew(0:\mbox{\tiny{2}},\mbox{\tiny{1}}|1:\mbox{\tiny{3}})}\ar^{2}[d]&&\mbox{\Skew(0:\mbox{\tiny{1}},\mbox{\tiny{2}}|1:\mbox{\tiny{3}})}\ar^{1}[d]\\
\mbox{\Skew(0:\mbox{\tiny{3}},\mbox{\tiny{1}}|1:\mbox{\tiny{3}})}\ar^{1}[dr]&&\mbox{\Skew(0:\mbox{\tiny{2}},\mbox{\tiny{2}}|1:\mbox{\tiny{3}})}\ar^{2}[dl]\\
&\mbox{\Skew(0:\mbox{\tiny{3}},\mbox{\tiny{2}}|1:\mbox{\tiny{3}})}&
}
&\hbox{        }&
\xymatrix{
&\mbox{\Skew(0:\mbox{\tiny{1}},\mbox{\tiny{2}},\mbox{\tiny{1}})}\ar^{1}[dl]\ar^{2}[dr]&\\
\mbox{\Skew(0:\mbox{\tiny{2}},\mbox{\tiny{2}},\mbox{\tiny{1}})}\ar^{2}[d]&&\mbox{\Skew(0:\mbox{\tiny{1}},\mbox{\tiny{3}},\mbox{\tiny{1}})}\ar^{1}[d]\\
\mbox{\Skew(0:\mbox{\tiny{2}},\mbox{\tiny{3}},\mbox{\tiny{1}})}\ar^{2}[d]\ar^{2}[d]&&\mbox{\Skew(0:\mbox{\tiny{1}},\mbox{\tiny{3}},\mbox{\tiny{2}})}\ar^{1}[d]\\
\mbox{\Skew(0:\mbox{\tiny{3}},\mbox{\tiny{3}},\mbox{\tiny{1}})}\ar^{1}[dr]&&\mbox{\Skew(0:\mbox{\tiny{2}},\mbox{\tiny{3}},\mbox{\tiny{2}})}\ar^{2}[dl]\\
&\mbox{\Skew(0:\mbox{\tiny{3}},\mbox{\tiny{3}},\mbox{\tiny{2}})}&
}
\end{align*}
\end{ex}

\subsection{Equivalence of galleries}
We say that a gallery of shape $\underline{d} = (d_{1}, \cdots, d_{r})$ is a \textit{semi-standard Young tableau} if $d_{1}\leq \cdots \leq d_{r}$ and if the entries are weakly increasing from left to right in rows. We will denote the set of all semi-standard Young tableaux of shape $\underline{d}$ by $\Gamma(\underline{d})^{\op{SSYT}}$.

\begin{ex}
The gallery
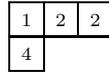
\begin{align*}
\Skew(0:\mbox{\tiny{1}}, \mbox{\tiny{2}}, \mbox{\tiny{2}}| 0: \mbox{\tiny{4}})
\end{align*}
is a semi-standard Young tableau. Note that the galleries considered in Example \ref{wordreadingexample} are not.  
\end{ex}

 We say that two galleries $\gamma, \delta$ are \textit{equivalent} ($\gamma \sim \delta$) if there exists a crystal isomorphism $\phi: \op{Conn}(\gamma) \tightoverset{\sim}{\longrightarrow} \op{Conn}(\delta)$ such that $\phi(\gamma) = \delta$.  The \textsl{plactic monoid} is the quotient $\mathcal{P}_{n} =\mathcal{W}_{n}/\sim$ of $\mathcal{W}_{n}$ by the ideal $\sim$ generated by the following relations. 

\begin{itemize}
\label{a}
\item[a.] For $x\leq y < z$, $y\hbox{ }x\hbox{ }z = y\hbox{ }z\hbox{ }x$.
\label{b}
\item[b.] For $x< y \leq z$, $x\hbox{ }z\hbox{ }y = x\hbox{ }y\hbox{ }z$.
\label{c}
\item[c.] The relation $1\hbox{ } \cdots \hbox{ } n = \o$, where $\o$ is the trivial word.
\end{itemize}

\noindent
If two words have equal classes in the plactic monoid, we say they are \textsl{plactic equivalent}. 

\begin{lem}
\label{knuthremark}
Two galleries $\delta$ and $\gamma$ are equivalent if and only if their words $w(\delta)$ and $w(\gamma)$ are plactic equivalent. 
\end{lem}
\begin{proof}
Let $\delta$ and $\gamma$ be two galleries, and assume that their words $w(\delta)$ and $w(\gamma)$ are plactic equivalent. Then by Main Theorem C b. in \cite{placticalgebra} this is equivalent to $\gamma_{w(\delta)} \sim \gamma_{w(\gamma)}$. Proposition \ref{wordreadingisacrystalmorphism} implies that word reading induces isomorphisms of crystals $\op{Conn}(\nu) \tightoverset{\sim}{\longrightarrow} w(\op{Conn}(\nu)), \nu \mapsto \gamma_{w(\nu)}$ for any gallery $\nu$, where $w(\op{Conn}(\nu))$ is the crystal of all words of elements in $\op{Conn}(\nu)$. This concludes the proof. 
\end{proof}

\begin{rem}
Lemma \ref{knuthremark} implies that our definition of equivalence of galleries coincides with Definition 5 in \cite{knuth} (after adding the relation $1\cdots n = \o$).
\end{rem}

\begin{rem}
The crystal structure we have defined coincides with the usual crystal structure on the set of semi-standard Young tableaux (see \cite{hongandkang}, section 7.4).
\end{rem}

The following lemma is well-known (originally Theorem 6 in \cite{knuthrelations}) and similar to Theorem 1 in \cite{knuth}, but note that we have an extra restriction on the length of the longest column of the galleries we consider. The reason for this is that we consider representation theory of $\op{SL}_{n}(\mathbb{C})$, where we have only $n-1$ fundamental weights.

\begin{lem}
\label{ss}
Given any gallery $\gamma$ there exists a unique semi-standard Young tableau $\gamma_{\op{SS}}$ such that $\gamma \sim \gamma_{\op{SS}}$. 
\end{lem}

\begin{proof}
Let $\gamma$ be a gallery and let $w$ be a representative of minimal length of the class in the plactic monoid $\mathcal{P}_{n}$ of its word $w(\gamma)$. Let $\gamma_{\op{SS}}$ be the semistandard Young tableau obtained by applying Robinson-Schensted-Knuth insertion (see for example \cite{schensted}, second definition in Part I) to $w$ read from right to left (the reason for this is that we want to keep the word reading convention of \cite{knuth}). Now we use a result of C. Schensted (Theorem 2 in \cite{schensted}, the general version from Part II): the number of rows (or the length of the longest column) of $\gamma_{\op{SS}}$ equals the length of the longest decreasing subsequence of $w$ (read from right to left!). Since $w$ is a minimal length representative, we claim that it cannot have an decreasing subsequence of length n. Indeed, for any $i\leq n$, the relations a. and b. above imply that for $j \leq i, 1\cdots i j$ and $j 1\cdots i$ are plactic equivalent. Hence if $w$ has a subsequence of length $n$ it follows by induction that it is plactic equivalent to a word of the form $w_{1}1\cdots n w_{2}$, which is plactic equivalent to $w_{1}w_{2}$. The word $w_{1}w_{2}$ has length strictly less than that of $w$, which contradicts the minimality assumption on the length of $w$. Schensted's result then implies that $\gamma_{\op{SS}}$ has columns of length at most $n-1$. By Theorem 6, \cite{knuthrelations}, $\gamma_{\op{SS}}$ is the unique semistandard Young tableau such that its word $w(\gamma_{\op{SS}})$ is plactic equivalent to $w$. The latter is in turn plactic equivalent to $w(\gamma)$; hence, by Lemma \ref{knuthremark}, $\gamma \sim \gamma_{\op{SS}}$.
\end{proof}

\begin{ex}
For $n=3$, the galleries $\Skew(0:\mbox{\tiny{1}},\mbox{\tiny{1}}|1:\mbox{\tiny{2}})$,
 $\Skew(0:\mbox{\tiny{1}},\mbox{\tiny{2}},\mbox{\tiny{1}})$, and 
  $\Skew(0:\mbox{\tiny{1}},\mbox{\tiny{2}},\mbox{\tiny{1}}, \mbox{\tiny{3}}, \mbox{\tiny{2}}, \mbox{\tiny{1}})$
  are all equivalent to the semi-standard Young tableau $\Skew(0:\mbox{\tiny{1}}, \mbox{\tiny{1}}|0: \mbox{\tiny{2}}).$ 
\end{ex}

\section{Galleries and MV cycles}
\subsection{Setup and notation}

For a $\mathbb{C}$-algebra $\A$ and an algebraic group $\G$ consider its $\A$-rational points $\op{G}(\A):= \op{Mor}_{k-alg}(\mathbb{C}[\op{G}], \A)$, where $\mathbb{C}[\op{G}]$ is the coordinate ring of $\G$. We shall make abuse of notation and write $\op{SL}_{n}(\A), \T_{\op{SL}_{n}}(\A), \op{PSL}_{n}(\A)$ and $\T_{\op{PSL}_{n}}(\A)$ for the $\A$-rational points of the groups that we consider. We refer the reader to Chapter 13 of \cite{kumar} for proofs of the statements in this subsection.\\

Consider the map $p: \op{SL}_{n}(\mathbb{C}) \rightarrow \op{PSL}_{n}(\mathbb{C}) = \op{SL}_{n}(\mathbb{C})/ \mu_{n}$, where $\mu_{n}$ is the group of n-th roots of unity.   The set $\mathcal{G} = \op{PSL}_{n}(\mathbb{C}((t)))/ \op{PSL}_{n}(\mathbb{C}[[t]])$ is the affine Grassmannian associated to $\op{PSL}_{n}(\mathbb{C})$, where $\mathbb{C}((t))$ and $\mathbb{C}[[t]]$ are the $\mathbb{C}$-algebras of formal power series and Laurent power series, respectively. It carries the structure of an ind-variety; this means it is the direct limit of projective varieties, and that all the maps are closed immersions. Each cocharacter $\lambda \in \op{Mor}(\mathbb{C}^{\times}, \T_{\op{PSL}_{n}}(\mathbb{C}))$ determines a $\mathbb{C}$-algebra morphism $\mathbb{C}[\op{PSL}_{n}] \rightarrow \mathbb{C}[\mathbb{C}^{\times}] = \mathbb{C}[t,t^{-1}] \subset \mathbb{C}((t))$; there is actually a bijection $\op{Mor}(\mathbb{C}^{\times}, \T_{\op{PSL}_{n}}(\mathbb{C})) \overset{1:1}{\longleftrightarrow} \T_{\op{PSL}_{n}}(\mathbb{C}((t)))/\T_{\op{PSL}_{n}}(\mathbb{C}[[t]])$. We will write $t^{\lambda}$ for the point in $\mathcal{G}$ determined by the cocharacter $\lambda$. \\

	The group $\op{SL}_{n}(\mathbb{C}((t)))$ acts on $\mathcal{G}$ naturally via the map $p': \op{SL}_{n}(\mathbb{C}((t))) \rightarrow \op{PSL}_{n}(\mathbb{C}((t)))$ that is induced by $p$. See Section 6 of \cite{knuth} for a more complete discussion of this. The $\op{SL}_{n}(\mathbb{C}[[t]])$-orbits in $\mathcal{G}$ coincide with the $\op{PSL}_{n}(\mathbb{C}[[t]])$-orbits and are parametrised by the dominant integral weights $\X^{+} \subset \op{Mor}(\T_{\op{SL}_{n}}(\mathbb{C}), \mathbb{C}^{\times}) \cong \op{Mor}(\mathbb{C}^{\times}, \T_{\op{PSL}_{n}}(\mathbb{C}))$. Explicitly: $$\mathcal{G} = \bigcup_{\lambda \in \X^{+}} \op{SL}_{n}(\mathbb{C}((t))) t^{\lambda}.$$  To each dominant integral weight $\lambda \in \X^{+}$ is associated a projective variety $\X_{\lambda} \subset \mathcal{G}$ that is defined as the closure of the $\op{SL}_{n}(\mathbb{C}[[t]])$-orbit of $t^{\lambda}$ in $\mathcal{G}$, with respect to its topology as an ind-variety. \\
	
\subsection{Bott-Samelson varieties}
In this section we write $\T$ for $\T_{\op{SL_{n}}}(\mathbb{C})$. Let $\B \supset \T$ be the Borel subgroup of upper triangular matrices in $\op{SL}_{n}(\mathbb{C})$ and let $\U$ be its unipotent radical. It is generated by the images $\U_{\alpha}(a)$ of the one parameter subgroups $\U_{\alpha}: \mathbb{C}^{\times} \rightarrow \op{SL}_{n}(\mathbb{C}), a \mapsto \op{Id} + a\E_{ij}$ associated to the roots $\alpha = \epsilon_{i}-\epsilon_{j} \in \Phi$. Note that $p(\U)$ and $\U$ are isomorphic. For $\lambda \in \X^{+}$ and $\mu \in \X$ such that $\mu \leq \lambda \hbox{  }(\lambda - \mu = \bigoplus_{i = 1}^{n-1}\mathbb{Z}_{\geq 0}\alpha_{i})$, let  $\mathcal{Z}(\lambda)_{\mu}$ be the set of irreducible components of the closure 
$$\overline{\U(\mathbb{C}((t)))t^{\mu}\cap \op{PSL}_{n}(\mathbb{C}[[t)]]t^{\lambda}} = \overline{\U(\mathbb{C}((t)))t^{\mu}\cap \op{SL}_{n}(\mathbb{C}[[t)]]t^{\lambda}}.$$ The geometric Satake equivalence implies that $\mathcal{Z}(\lambda)_{\mu}$ can be identified with a basis for the $\mathbb{C}$-vector space $\li(\lambda)_{\mu}$ (see Corollary 7.4 in \cite{mirkovicvilonen}). The set $\mathcal{Z}(\lambda) = \bigcup_{\mu \leq \lambda}\mathcal{Z}(\lambda)_{\mu}$ is  the set of \textsl{MV cycles} in $\X_{\lambda}$; see Theorem 3.2 of \cite{mirkovicvilonen}. This set carries the structure of a crystal isomorphic to $\B(\lambda)$ \cite{bravermangaitsgory}.  For each $i\in \{1, \cdots, n-1\}$ we denote by $\tightoverset{\sim}{{e}_{i}}, \tightoverset{\sim}{{f}_{i}}$ the crystal operators on the set $\mathcal{Z}(\lambda)$ defined by Braverman and Gaitsgory  (see Section 3.3 of  \cite{bravermangaitsgory}).\\

To each shape $\underline{d}$ we assign the dominant integral weight $\lambda_{\underline{d}} = \omega_{d_{1}}+ \cdots + \omega_{d_{r}} \in \X^{+}$ as well as an affine Bott-Samelson desingularization $\Sigma_{\underline{d}} \overset{\pi_{\underline{d}}}{\longrightarrow} \X_{\lambda_{\underline{d}}}$ that is defined as follows. Let $\U_{\alpha,n}:\mathbb{C}^{\times}\rightarrow \op{SL}_{n}(\mathbb{C}((t)))$ be the one parameter subgroup defined by $b \mapsto \U_{\alpha}(bt^{n})$. For $i \leq r$ let  $\mu_i := \sum_{j\leq i}\omega_{d_{r-j+1}}$ and let $l_{i}$ be the line segment that joins $\mu_{i}$ and $\mu_{i+1}$. Let 

 $P_{d_{i}}$ be the subgroup of $\op{SL}_{n}(\mathbb{C}((t)))$ that is generated by the elements $\U_{\alpha,n}(b)$ for $b \in \mathbb{C}$, and such that $\mu_{i} \in \mm^{+}_{\alpha, n}$, and let $Q_{d_{i}}$ be the subgroup of  $P_{d_{i}}$  generated by the elements of the root subgroups for roots $(\alpha,n)$ such that the line segment $l_{i}$ joining $\mu_{i}$ and $\mu_{i+1}$ is contained in the corresponding hyperplane $l_{i} \subset \mm^{+}_{\alpha, n}$.\\

  The affine Bott-Samelson variety is defined as the quotient $\Sigma_{\underline{d}} := P_{0}\times \cdots \times P_{r} / Q_{0}\times \cdots \times Q_{r-1}  \times P_{r}$ of $P_{0}\times \cdots \times P_{r}$ by the left  action of the group $P_{r} \ Q_{0}\times \cdots \times Q_{r-1}  \times P_{r}$ given by:
$$
(q_{0}, \cdots, q_{r})\cdot (p_{0}, \cdots, p_{r}) = (p_{0}q_{0}, q_{0}^{-1}p_{1}q_{1}, \cdots, q_{r-1}^{-1}p_{r}q_{r}).
$$
It is well known that the quotient $\Sigma_{\underline{d}}$ is a smooth projective variety and that the map $\Sigma_{\underline{d}} \overset{\pi_{\underline{d}}}{\longrightarrow} \X_{\lambda_{\underline{d}}}$ defined by $[g_{0}, \cdots, g_{r}] \mapsto g_{0}\cdots g_{r-1}t^{\lambda_{\underline{d}}}$ has image $\X_{\lambda}$ and is a desingulatization \cite{onesk}.  The maximal torus $\T_{\op{SL}_{n}}(\mathbb{C})$ acts by multiplication on the left-most coordinate. The choice of a generic dominant coweight $\eta : \mathbb{C}^{\times} \rightarrow \T_{\op{SL}_{n}}(\mathbb{C})$ induces a $\mathbb{C}^{\times}$-action with set of fixed points in bijection with the set $\Gamma(\underline{d})$ of galleries of shape $\underline{d}$ (see Lemma 1 in \cite{knuth}). Given a gallery $\gamma \in \Gamma(\underline{d})$ we denote its corresponding Bialynicki-Birula cell by $\C_{\gamma} := \{x \in \Sigma_{\underline{d}}: \underset{t\rightarrow 0}{\op{lim }}\hbox{  }\eta(t)\cdot x = \gamma \} \subset \Sigma_{\underline{d}}$. One of the main results in \cite{knuth} (Theorem 2) establishes that the closure of the image $\pi_{\underline{d}}(\C_{\gamma})$ is an MV cycle in $\mathcal{Z}(\lambda_{\underline{d}(\gamma_{\op{SS}})})$, where  $\gamma_{\op{SS}}$ is the the semi-standard Young tableau $\gamma_{\op{SS}}$ from Lemma \ref{ss} associated to $\gamma$.

\subsection{Galleries and MV cycles}
The following theorem is the combination of Theorem 2 in \cite{ls} and Section 6 in \cite{onesk} for part a., and Theorem 25 in \cite{baumanngaussent} for part b..

\begin{thm}
\label{baumanngaussentlittelmann}
Let $\underline{d} = (d_{1}, \cdots, d_{r})$ be a shape such that $d_{1}\leq \cdots \leq d_{r}$ and consider the desingularization $\pi_{\underline{d}}: \Sigma_{\underline{d}} \rightarrow \X_{\lambda_{\underline{d}}}$. 
\begin{itemize}
\item[a.] If $\delta \in \op{SSYT}(\lambda_{\underline{d}})$ is a semi-standard Young tableau, the closure $\overline{\pi_{\underline{d}}(\C_{\delta})}$ is an MV cycle in $\mathcal{Z}(\lambda_{\underline{d}})$. This induces a bijection $\op{SSYT}(\lambda_{\underline{d}}) \overset{\varphi_{\underline{d}}}{\longrightarrow} \mathcal{Z}(\lambda_{\underline{d}})$. 
\item[b.] The bijection $\varphi_{\underline{d}}$ is a morphism of crystals. 
\end{itemize}
\end{thm}

\begin{rem}
The set of one-skeleton LS galleries considered in \cite{onesk} coincides with the set of semi-standard Young tableaux (see Proposotion 18 $i.$ of \cite{onesk}).
\end{rem}

Let $\underline{d}$ be a shape. For $\lambda \in \X^{+}$, let $n_{\underline{d}}^{\lambda} = \# \{\gamma \in \Gamma(\underline{d})^{\op{dom}}: \lambda_{\underline{d}(\gamma)} = \lambda \}$ and let $\X^{+}_{\underline{d}}:= \{\lambda \in \X^{+}: n^{\lambda}_{\underline{d}} \neq 0\}$. Here $\Gamma(\underline{d})^{\op{dom}}$ is the set of all dominant galleries of shape $\underline{d}$. Fix $\lambda = \lambda_{1}\omega_{1}+\cdots + \lambda_{n-1}\omega_{n-1}$ and $\Z \in \mathcal{Z}(\lambda)_{\mu}$ for some $\mu \leq \lambda$. By Theorem \ref{baumanngaussentlittelmann} there exists a unique semi-standard Young tableau $\gamma^{\lambda}_{\mu, \Z} \in \op{SSYT}(\underline{\lambda}) \hbox{ of shape }\underline{\lambda}:= (d^{\lambda}_{1}, \cdots, d^{\lambda}_{k_{\lambda}})$, where $k_{\lambda}:= \sum_{i=1}^{n-1}\lambda_{i}$ and $d^{\lambda}_{j}:= i$ for $\lambda_{i-1} < j \leq \lambda_{i}, \lambda_{0}:=0$, such that $\varphi_{\underline{\lambda}}(\gamma^{\lambda}_{\mu, \Z}) = \Z$.

\begin{thm}
\label{main}
\begin{itemize}
\item[a.] The map 
\begin{align*}
\Gamma(\underline{d})& \overset{\varphi_{\underline{d}}}{\longrightarrow} \underset{\lambda \in \X^{+}_{\underline{d}}}{\bigoplus}\mathcal{Z}(\lambda)\\
\delta & \longmapsto \overline{\pi_{\underline{d}}(\C_{\delta})}
\end{align*}
is a well-defined surjective morphism of crystals. 
\item[b.] If $\C$ is a connected component of $\Gamma(\underline{d})$, the restriction $\varphi_{\underline{d}}|_{\C}$ is an isomorphism onto its image.  
\item[c.] The number of connected components $\C$ of $\Gamma (\underline{d})$ such that 
$\varphi_{\underline{d}}(\C) = \mathcal{Z}(\lambda)$ (for $\lambda \in \X^{+}_{\underline{d}}$) is equal to $n^{\lambda}_{\underline{d}}$.
\item[d.] The fibre $\varphi_{\underline{d}}^{-1}(\Z)$ is given by 
$$
\varphi_{\underline{d}}^{-1}(\Z) =    \{\gamma \in \Gamma(\underline{d}): \varphi_{\underline{d}}(\gamma)= \Z\} =  \{\gamma \in \Gamma(\underline{d}): \gamma \sim \gamma^{\lambda}_{\mu, \Z}\}
.$$
\end{itemize}
 We consider the direct sum ${\bigoplus}_{\lambda \in \X^{+}_{\underline{d}}}\mathcal{Z}(\lambda)$ in the category of crystals, regarding the $\mathcal{Z}(\lambda)$ as abstract crystals. 
\end{thm}

\begin{proof}

Let $\underline{d}$ be a shape and $\delta \in \Gamma(\underline{d})$ as in the statement of the Theorem. By Lemma \ref{ss} there exists a unique semi-standard Young tableau $\delta_{SS}$  such that $\delta \sim \delta_{SS}$. By Theorem \ref{thmGLN} b. (Theorem 2 b. in \cite{knuth} up to a small correction, see the Appendix) and Lemma 2, 

\begin{align}
\label{proofeq1}
\overline{\pi_{\underline{d}}(\C_{\delta})} = \overline{\pi_{\underline{d}(\delta_{SS})}(\C_{\delta_{SS}})}.
\end{align}

Now let $r$ be a root operator. By definition of equivalence of galleries $r(\delta) \sim r(\delta_{SS})$. Note also that $\underline{d}(r(\delta)) = \underline{d}$ and $\underline{d}(r(\delta_{SS})) = \underline{d}(\delta_{SS})$. Lemma 2 and Theorem \ref{thmGLN} b. again imply 

\begin{align*}
\overline{\pi_{\underline{d}}(\C_{r(\delta)})} = \overline{\pi_{\underline{d}(\delta_{SS})}(\C_{r(\delta_{SS})})}.
\end{align*}
\noindent
Theorem \ref{baumanngaussentlittelmann} b. says that
\begin{align*}
 \overline{\pi_{\underline{d}(\delta_{SS})}(\C_{r(\delta_{SS})})} = \tightoverset{\sim}{r}(\overline{\pi_{\underline{d}(\delta_{SS})}(\C_{\delta_{SS}})}),
 \end{align*}
 \noindent
  and since (\ref{proofeq1}) implies $\tightoverset{\sim}{r}(\overline{\pi_{\underline{d}}(\C_{\delta})}) = \tightoverset{\sim}{r}(\overline{\pi_{\underline{d}(\delta_{SS})}(\C_{\delta_{SS}})})$, the proof of part a. of Theorem \ref{main} is complete.\\

Parts b., c., and d. are a direct consequence of Theorem \ref{pathmodel}: Indeed, since the action of the root operators does not affect the shape of a gallery, Theorem \ref{pathmodel} implies that the set $\Gamma(\underline{d})$ is a disjoint union 
$\Gamma(\underline{d})  = \bigsqcup_{\eta \in \Gamma(\underline{d})^{\op{dom}}} \op{Conn}(\eta)$. The above argument and Theorem \ref{baumanngaussentlittelmann} imply that $\varphi_{\underline{d}}(\op{Conn}(\eta)) = \mathcal{Z}(\op{wt}(\eta))$ for $\eta \in \Gamma(\underline{d})^{\op{dom}}$ and that $\varphi_{\underline{d}}$ is a crystal isomorphism onto its image when restricted to $\op{Conn}(\eta)$. 
\end{proof}

\section{Appendix}
\label{appendix}
Here we state Theorem 2 in \cite{knuth} with a small correction, which we prove. What is missing in the formulation given in \cite{knuth} is the relation $1\hbox{ } \cdots \hbox{ } n = \o$. The proof we provide shows the failure of Theorem \ref{thmGLN} without it. 

\begin{thm}
\label{thmGLN}
Let $\gamma$ be a gallery of shape $\underline{d}$, and let $\gamma_{SS}$ be the unique semistandard Young tableau such that the words $w(\gamma)$ and $w(\gamma_{SS})$ are plactic equivalent. Let $\underline{c}$ be the shape of $\gamma_{SS}$. Consider the Schubert varieties $\X_{\lambda_{\underline{c}}} \subset \X_{\lambda_{\underline{d}}}$ and the desingularizations $\pi_{\underline{d}}: \Sigma_{\underline{d}} \rightarrow \X_{\lambda_{\underline{d}}}$ and $\pi_{\underline{c}}: \Sigma_{\underline{c}} \rightarrow \X_{\lambda_{\underline{c}}}$.
\begin{itemize}
\item[a.] The closure $\overline{\pi_{\underline{d}}(\C_{\gamma})} \subset \X_{\lambda_{\underline{d}}}$ is an MV cycle in $\mathcal{Z}(\lambda_{\underline{c}})$.
\item[b.] Let $\gamma'$ be a second gallery of shape $\underline{d}'$. Then $\gamma \sim \gamma'$ if and only if $\overline{\pi_{\underline{d}}(\C_{\gamma})} = \overline{\pi_{\underline{d'}}(\C_{\gamma'})}$.
\end{itemize}
\end{thm}

For the proof we need the following description of the image $\pi_{\underline{d}}(\C_{\gamma})$. Let $\gamma$ be a gallery of shape $\underline{d} = (d_{1}, \cdots, d_{r})$. Assume that the boxes of column $i$ (read from right to left) are filled in with integers $1\leq l^{i}_{1}< \cdots <l^{i}_{r_{i}} \leq n$. Define $\gamma_{0}:=0$ and let $l_{1}$ be the line segment that joins the origin and the point $\gamma_{1}:= \varepsilon_{l^{1}_{1}}+ \cdots +\varepsilon_{l^{1}_{r_{1}}}$. Define $\gamma_{j+1}:= \gamma_{j}+ \varepsilon_{l^{j+1}_{1}}+ \cdots \varepsilon_{l^{j+1}_{r_{j+1}}}$ recursively; $l_{j+1}$ is the line segment joining $\gamma_{j}$ and $\gamma_{j+1}$. Let $\Phi_{\gamma_i}^{+}:= \{(\alpha, n) \in \Phi^{+}\times \mathbb{Z} : \gamma_{i} \in \mm^{+}_{\alpha, n}\}$ and $\Phi_{\gamma_i,\gamma_{i+1}}^{+}=\{(\alpha, n) \in \Phi_{\gamma_{i}}^{+}:  \l_{i} \nsubset \mm^{-}_{\alpha, n}\}$. Fix some total order on $\Phi_{\gamma_i,\gamma_{i+1}}^{+}$ and let $\mathbb{U}_{\gamma_i}(\underline{a^{i}}):= \underset{(\alpha, n) \in \Phi^{+}_{\gamma_{i},\gamma_{i+1}}}{\prod}\U_{\alpha,n}(a^{i}_{\alpha,n})$, where $\underline{a^{i}} := (a^{i}_{\alpha, n})_{(\alpha,n) \in \Phi^{+}_{\gamma_{i},\gamma_{i+1}}} \in \mathbb{C}^{ \# \Phi^{+}_{\gamma_{i},\gamma_{i+1}}}$ and the product is taken in the chosen fixed order.  Then Proposition 4.19 in \cite{onesk} (or Corollary 3 in \cite{knuth}) says:

\begin{align*}
\pi(\C_{\gamma}) = \{\mathbb{U}_{\gamma_0}(\underline{a^{0}})\cdots \mathbb{U}_{\gamma_{r-1}}(\underline{a^{r-1}})t^{\op{wt}(\gamma)}: \underline{a^{j}} \in \mathbb{C}^{\# \Phi^{+}_{\gamma_{i},\gamma_{i+1}}}\}.
\end{align*}

Consider also, for $0 \leq k\leq r-1$ the \textsl{truncated images}

\begin{align*}
\T_{\gamma}^{\geq k} &: = \mathbb{U}_{\gamma_{k}}(\underline{a^{k}})\cdots \mathbb{U}_{\gamma_{r-1}}(\underline{a^{r-1}})t^{\op{wt}(\gamma)}\\
\T_{\gamma}^{\leq k} &: = \mathbb{U}_{\gamma_0}(\underline{a^{0}})\cdots \mathbb{U}_{\gamma_{k}}(\underline{a^{k}}).\\
\end{align*}

Given a weight $\mu \in \X$ and $(\alpha, n) \in \Phi \times \mathbb{Z}$, the relation $t^{-\mu}\U_{\alpha,n}t^{\mu} = \U_{\alpha, n - (\lambda, \alpha)}$ (see \cite{steinberg}, Section 6) implies that the group $\U_{\mu}$ generated by all the subgroups $\U_{\beta, m}$ such that $\lambda \in \mm^{-}_{\beta,m}$ stabilises $t^{\mu}$, and by Proposition 3 in \cite{knuth}, $\U_{\op{wt}(\gamma^{ < k})}\T_{\gamma}^{\geq k} = \T_{\gamma}^{\geq k}$, where $\gamma^{<k}$ is the gallery consisting of the first $k-1$ columns of $\gamma$, read from right to left. We will use this below. 

\begin{proof}[Proof of Theorem \ref{thmGLN}]
The only thing missing in the proof in \cite{knuth} is the following claim.
\begin{cl}
 Let $\gamma$ and $\delta$ be galleries, let $\underline{b}$ be the shape of $\gamma * \delta$, and $\underline{a}$ be the shape of $\gamma *\gamma_{1\cdots n}*\delta$. Then $\overline{\pi_{\underline{a}}(\C_{\gamma *\gamma_{1\cdots n}*\delta})} = \overline{\pi_{\underline{b}}(\C_{\gamma *\delta})}$. 
 \end{cl}

Let $\eta = \gamma *\gamma_{1\cdots n}*\delta$, and assume $\gamma$ has $k$ columns. Then $\pi(\C_{\eta}) = \T_{\eta}^{\leq k} \mathbb{U}_{\eta_{k+1}}(\underline{a^{k+1}}) \cdots \mathbb{U}_{\eta_{k+n}}(\underline{a^{k+n}}) \T_{\eta}^{\geq k+n+1}$. Now, note that the sets $\Phi^{+}_{\eta_{k+i, k+i+1}}$ are disjoint for $i \in \{1, \cdots, n\}$ (for example, if $\gamma$ and $\delta$ are trivial and $n=3$ then $\Phi^{+}_{\eta_{0,1}} = \{(\varepsilon_{1}-\varepsilon_{2},0), (\varepsilon_{1}-\varepsilon_{3},0)\}, \Phi^{+}_{\eta_{1,2}} = \{(\varepsilon_{2}-\varepsilon_{3},0)\}, \Phi^{+}_{\eta_{2,3}} = \emptyset$). This implies that for $i \in \{1, \cdots, n\}$ the products $\mathbb{U}_{\eta_{k+i}}(\underline{a^{k+i}})$ all belong to the group $\U_{\op{wt}(\gamma)}$, which stabilises $\T_{\eta}^{\geq k+n+1}$. Since $\op{wt}(\gamma_{1\cdots n}) = 0, \T_{\eta}^{\geq k+n+1} = \T_{\gamma * \delta}^{\geq k+1}$.  Hence $\pi(\C_{\eta}) = \T_{\eta}^{\leq k} \T_{\eta}^{\geq k+n+1} = \T_{\gamma * \delta}^{\leq k}\T_{\gamma * \delta}^{\geq k+1} = \pi_{\underline{b}}(\C_{\gamma *\delta})$ and the claim follows. 
%

\end{proof}

\end{document}